\documentclass[12pt,leqno]{amsart}
\usepackage{amssymb,amsthm,amsmath,latexsym}
\newtheorem{theorem}{\sc Theorem}[section]
\newtheorem{thm}[theorem]{\sc Theorem}
\newtheorem{lem}[theorem]{\sc Lemma}

\newtheorem{cor}[theorem]{\sc Corollary}

\newtheorem{ex}[theorem]{\sc Example}

 \newtheorem*{thmA}{Theorem A}
 \newtheorem*{thmB}{Theorem B}
 \newtheorem*{thmC}{Theorem C}
 
 \DeclareMathOperator{\PSL}{PSL}
 \DeclareMathOperator{\FC}{FC} 
 \DeclareMathOperator{\BFC}{BFC}  
 \DeclareMathOperator{\spec}{spec}
 
 \DeclareMathOperator{\ASL}{ASL}
 \DeclareMathOperator{\SL}{SL}
 
 \DeclareMathOperator{\Aut}{Aut} 
 \DeclareMathOperator{\Tor}{Tor}  
 \DeclareMathOperator{\cent}{Z} 
 \DeclareMathOperator{\centhat}{\hat{Z}} 
 \DeclareMathOperator{\centralizer}{C_G} 
 \DeclareMathOperator{\Alt}{A}
 \DeclareMathOperator{\Sym}{S} 
 \DeclareMathOperator{\GH}{GH}
 \DeclareMathOperator{\V}{V}
 \DeclareMathOperator{\N}{N}  
\title[Automorphism Orbits]{FC-groups with finitely many automorphism orbits}
\author[Bastos]{Raimundo A. Bastos}
\address{(Bastos) Departamento de Matem\'atica, Universidade de Bras\'ilia,
Brasilia-DF, 70910-900 Brazil}
\email{bastos@mat.unb.br}
\author[Dantas]{Alex C. Dantas}
\address{(Dantas) Departamento de Matem\'atica, Universidade de Bras\'ilia,
Brasilia-DF, 70910-900 Brazil}
\email{alexcdan@gmail.br}


\begin{document}
\maketitle

\begin{abstract}
Let $G$ be a group. The orbits of the natural action of $\Aut(G)$ on $G$ are called ``automorphism orbits'' of $G$, and the number of automorphism orbits of $G$ is denoted by $\omega(G)$. In this paper we prove that if $G$ is an $\FC$-group  with finitely many automorphism orbits, then the derived subgroup $G'$ is finite and $G$ admits a decomposition $G = \Tor(G) \times D$, where $\Tor(G)$ is the torsion subgroup of $G$ and $D$ is a divisible characteristic subgroup of $\cent(G)$. We also show that if $G$ is an infinite $\FC$-group with $\omega(G) \leqslant 8$, then either $G$ is soluble or $G \cong \Alt_5 \times H$, where $H$ is an infinite abelian group with $\omega(H)=2$. Moreover, we describe the structure of the infinite non-soluble $\FC$-groups with at most eleven automorphism orbits.   
\end{abstract}

\maketitle

\section{Introduction}
An element $g$ of a group $G$ is called an $\FC$-element if it has a finite number of conjugates in $G$, that is, if the index $[G : \centralizer(g)]$ is finite, where $\centralizer(g) = \{h \in G \mid  g^{h} = g\}$.  The set $\centhat(G)$ of all $\FC$-elements of $G$ is a characteristic subgroup (see \cite[14.5.5]{Rob}). If each element of $G$ is an $\FC$-element, then $G$ is called an $\FC$-group. In particular, $G = \centhat(G)$. If there exists a positive integer $d$ such that $[G : \centralizer(g)] \leq d$ for all $g \in G$, then the group $G$ is called a $\BFC$-group, that is, there is a common bound for the size of conjugacy classes in $G$. The class of $\FC$-groups contains all abelian and all finite groups. 

Let $G$ be a group. The orbits of the natural action of $\Aut(G)$ on $G$ are called ``automorphism orbits'' of $G$, and the number of automorphism orbits of $G$ is denoted by $\omega(G)$.  It is interesting to ask what can we
say about ``$G$'' only knowing $\omega(G)$. It is obvious that $\omega(G)=1$ if and only if $G=\{1\}$, and it is well known that if $G$ is a finite group then $\omega(G) = 2$ if and only if $G$ is elementary abelian. In \cite{LD}, Laffey and MacHale proved that if $G$ is a finite non-soluble group with $\omega(G) \leqslant 4$, then $G$ is isomorphic to $\PSL(2,\mathbb{F}_4)$. Stroppel in \cite{S1} has shown that the only finite non-abelian simple groups $G$ with $\omega(G) \leq 5$ are the groups $\PSL(2,\mathbb{F}_q)$ with $q \in \{4,7,8,9\}$. In \cite{BD,BDG}, the authors prove that if $G$ is a finite non-soluble group with $\omega(G) \leq 6$, then $G$ is isomorphic to one of $\PSL(2,\mathbb{F}_q)$ with $q \in \{4,7,8,9\}$, $\PSL(3,\mathbb{F}_4)$ or $\ASL(2,\mathbb{F}_4)$ (answering a question of Stroppel, cf. \cite[Problem 2.5]{S1}). Here $\ASL(2,\mathbb{F}_4)$ is the affine group $\SL(2,\mathbb{F}_4) \ltimes \mathbb{F}_4^2$ where $\SL(2,\mathbb{F}_4)$ acts naturally on $\mathbb{F}_4^2$. For more details concerning automorphism orbits of finite groups see \cite{K02,K04,S99,Z}.  

Some aspects of automorphism orbits are also investigated for infinite groups. In \cite{MS1}, M\"aurer and Stroppel classified the groups with a nontrivial characteristic subgroup and with three orbits by automorphisms (see Section 5, below). Schwachh\"ofer and Stroppel in \cite{S2}, have shown that if $G$ is an abelian group with finitely many automorphism orbits, then $G = \Tor(G) \oplus D$, where $D$ is a characteristic torsion free divisible subgroup of $G$ and $\Tor(G)$ is the set of all torsion elements in $G$. Stroppel also investigated properties of topological groups with finitely many automorphism orbits (see \cite{S1} for more details).  

In the present article we consider infinite $\FC$-groups with finitely many automorphism orbits. Our results can be viewed as extensions of the above results.

\begin{thmA} 
Let $G$ be an infinite $\FC$-group with finitely many automorphism orbits. Then the derived subgroup $G'$ is finite. Moreover, $G$ admits a decomposition $\Tor(G) \times D$, where $D$ is a divisible characteristic subgroup of $\cent(G)$.
\end{thmA}

In general, infinite groups with small number of automorphism orbits need not be soluble. For instance, Higman, Neumann and Neumann have constructed a torsion free non-abelian simple group $S$ with $\omega(S) = 2$ (cf.  \cite[6.4.6]{Rob}). However, we prove the following solubility criterion for infinite $\FC$-groups in terms of the number of automorphism orbits.

\begin{thmB} \label{thm.soluble}
Let $G$ be an infinite $\FC$-group. 
\begin{itemize}
\item[$(a)$] If $\omega(G) \leqslant 4$, then $G$ is nilpotent;
\item[$(b)$] If $\omega(G) \leqslant 7$, then $G$ is soluble;
\item[$(c)$] If $G$ is a non-soluble group with $\omega(G) = 8$, then there exists an infinite abelian group $H$ with $\omega(H)=2$ such that $G$ is isomorphic to $\Alt_5 \times H$.
\end{itemize}
\end{thmB}

Note that the group $\Sym_3 \times H_2$ is an infinite non-nilpotent $\FC$-group with $\omega(\Sym_3 \times H_2)=5$, where $\Sym_3$ is the symmetric group and $H_2$ is an infinite elementary abelian $2$-group (see Example \ref{ex.5}, below). Consequently, the bound on the number of orbits in Theorem B (a) is sharp. It should be noted that $\omega(\Alt_5 \times \mathbb{Q})=8$, where $\Alt_5$ is the alternating group and $\mathbb{Q}$ is the additive group of the field of rational numbers. In particular, the bound obtained in Theorem B (b) cannot be improved (see Example \ref{ex.5}, below). 

The famous Schur Theorem says that if $G$ is central-by-finite, then the order of the derived subgroup $G'$ is finite (cf. \cite[10.1.4]{Rob}). In particular, $G$ is a $\BFC$-group. Later, Neumann proved that the group $G$ is a $\BFC$-group if and only if the derived subgroup $G'$ is finite (cf. \cite[14.5.11]{Rob}). It is well known that $\BFC$-groups need not be central-by-finite (cf. \cite{N}). By Theorem B, every infinite non-soluble $\FC$-group with at most eight automorphism orbits is central-by-finite. We obtain the following related result.   

\begin{thmC}
Let $G$ be an infinite non-soluble $\FC$-group with at most eleven automorphism orbits. Then $G$ is central-by-finite.  
\end{thmC}

The above theorem is no longer valid if the assumption of non-solubility of $G$ is dropped (see Example \ref{ex.H}, below). Moreover, the bound obtained in the above result cannot be improved (see Example \ref{ex.12}, below). 

\section{Proof of Theorem A}

Let $G$ be a group and $g$ an element in $G$. The {\it automorphism orbit} of $g$ in the group $G$ is denoted by $g^{\Aut(G)} =\{g^{\alpha} \mid  \alpha \in \Aut(G)\}$. If $K$ is a subgroup of $G$, the notation $K = \uplus_{i=1}^{n} g_{i}^{\Aut(G)}$ means that $$K = \displaystyle \bigcup_{i=1}^{n} g_{i}^{\Aut(G)}$$ and $g_i^{\Aut(G)} \cap g_j^{\Aut(G)} = \emptyset$ if $i \neq j$.

\begin{lem} \label{lemma2}
(Stroppel, \cite[Lemma 3.1]{S}) Let $G$ be a nontrivial group with finitely many automorphism orbits. Suppose that there exist nontrivial characteristic subgroups $M$ and $N$ such that $M \cap N = 1$. Then 
$$\omega(G) \geqslant \omega(M)\omega(N) \geqslant 2 \cdot \omega(N).$$
Moreover, if $G = M \times N$, then $\omega(G) = \omega(M) \cdot \omega(N)$.  
\end{lem}

\begin{lem} (Stroppel, \cite[Lemma 1.1]{S2} \label{Stroppel1})
If $G$ is an abelian group with finitely many automorphism orbits, then $G = \Tor(G) \oplus D$, where $D$ is a characteristic torsion free divisible subgroup of $G$. In particular, $D$ is the additive group of a vector space over $\mathbb{Q}$.
\end{lem}

Schur \cite[10.1.4]{Rob} has shown that if $G$ is a central-by-finite group then the derived subgroup $G'$ is finite. In particular, $G$ is a $\BFC$-group. It is well known that the converse of Schur's Theorem does not hold (see \cite{N}). However, Hall \cite[14.5.3]{Rob} has shown that if $G$ is a group whose the $k$-th term of the lower central series $\gamma_{k+1}(G)$ is finite, then $\cent_{2k}(G)$. The next result is an immediate consequence of Hall's result  in the context of infinite BFC-groups. We supply the proof for the reader's convenience. 

\begin{lem} \label{lem.center}
Let $G$ be an infinite $\FC$-group with finitely many automorphism orbits. Then the center $\cent(G)$ is non-trivial.  
\end{lem}

\begin{proof}
By Theorem A, the derived subgroup $G'$ is finite. By Hall's theorem \cite[14.5.3]{Rob}, the second center $\cent_2(G)$ has finite index in $G$. Since $G$ is infinite, it follows that the center $\cent(G)$ is non-trivial. The proof is complete.   
\end{proof}

\begin{lem} \label{lem.abelian}
Let $G$ be an infinite $\FC$-group with $\omega(G) = 2$. Then $G$ is the additive group of a vector space, that is, either a torsion free divisible abelian group or an elementary abelian $p$-group for some prime $p$.     
\end{lem}

\begin{proof}
By Lemma \ref{lem.center}, the center $\cent(G)$ is nontrivial. It follows that $G = \cent(G)$. In particular, $G$ is the additive group of a vector space over some field $\mathbb{F}$. The result follows.
\end{proof}

\begin{lem} \label{lemma1}
Let $G$ be an $\FC$-group with finitely many automorphism orbits. Then $G$ is a $\BFC$-group.
\end{lem}

\begin{proof}
The size of the conjugacy class $g^G$ is constant on each orbit under $\Aut(G)$. If $\omega(G)$ is finite then there is a finite bound for those sizes.
\end{proof}

We are now in a position to prove Theorem A. 

\begin{proof}[Proof of Theorem A] Assume that $G$ is an $\FC$-group with finitely many automorphism orbits. We claim that $G$ is a $\BFC$-group and admits a decomposition $D \times \Tor(G)$, where $D$ is a divisible characteristic subgroup of $\cent(G)$.   

By Lemma \ref{lemma1}, $G$ is a BFC-group. By Neumann's result \cite[14.5.11]{Rob}, the derived subgroup $G'$ is finite. It remains to prove the decomposition $G = \Tor(G) \times D$. By Lemma \ref{Stroppel1}, $\cent(G) = (\Tor(G) \cap \cent(G)) \oplus D$, where $D$ is a characteristic subgroup of $\cent(G)$. By Baer's result \cite[14.5.6]{Rob}, $G/\cent(G)$ is a torsion group. Since $\omega(G)$ is finite, we conclude that $G / \cent(G)$ has finite exponent. Set $m= \exp(G/\cent(G))$. Since $G$ is an $\FC$-group, it follows that $\Tor(G)$ is a subgroup of $G$ (cf. \cite[14.5.9]{Rob}).  By Lemma \ref{Stroppel1}, $G/\Tor(G)$ is a divisible group. Let $g \in G \setminus \Tor(G)$. Hence there exists $h \in G$ such that $g = h^{m}k$, where $k \in \Tor(G)$ and thus $G = \Tor(G)\cent(G) = \Tor(G)D$. On the other hand, $G = \Tor(G) \times D$, because $D$ is torsion free. The proof is complete.  
\end{proof}

\begin{cor} \label{cor.odd}
Let $G$ be an infinite $\FC$-group in which $\omega(G)$ is an odd number. Then $G$ has finite exponent.  
\end{cor}

\begin{proof}
According to Theorem A, $G = \Tor(G) \times D$, where $D$ is a divisible characteristic subgroup of $G$. Suppose that $D$ is non-trivial. By Lemma \ref{Stroppel1}, $\omega(D) = 2$. By Lemma \ref{lemma2},   $$\omega(G) = \omega(\Tor(G))\omega(D) = 2 \omega(\Tor(G)),$$ a contradiction. The proof is complete.  
\end{proof}

\begin{cor} \label{Cor.na.5}
Let $G$ be a non-abelian $\FC$-group with $\omega(G) \leq 5$. Then $G$ has finite exponent.
\end{cor}

\begin{proof}
By Corollary \ref{cor.odd}, we can assume that $G$ has exactly four automorphism orbits. Suppose that $G$ is not a torsion group. By Theorem A, $G = D \times \Tor(G)$, where $D$ is a torsion free divisible group and so, $G$ is abelian.    
\end{proof}

\begin{ex}
Note that the group $\Sym_3 \times \mathbb{Q}$ is an infinite non-periodic $\FC$-group with $\omega(\Sym_3 \times \mathbb{Q}) = 6$ (Lemma \ref{lemma2}). Thus, the bound on the number of orbits in the above result cannot be improved.
\end{ex}

\section{Proof of Theorem B}

Let $H$ be a periodic group and $S$ a subset of $H$. The {\it spectrum} of $S$, denoted by $\spec(S)$, is defined to be the set of all orders of elements in $S$. A crucial observation that we will use many times is that two elements that lie in the same automorphism orbit have the same order. In particular, it is clear that if $G$ is a $p$-group with finitely many automorphism orbits, then $p^{\omega(G)} \geqslant \exp(G)$. The following result is a key argument to obtain bounds for the exponent $\exp(G)$, when $G$ is an infinite non-abelian $\BFC$-group.    

\begin{lem} \label{Lemma 4.2}
Let $G$ be an infinite $\FC$-group with finitely many automorphism orbits. Assume that $G$ is a non-abelian $p$-group. Then there exists $h \in G'$ and $g \in G \setminus G'$ such that $|h| = |g|$.
\end{lem}

\begin{proof}
Since $G$ is a $p$-group with finitely many automorphism orbits, we conclude that the exponent $\exp(G)$ is finite. Suppose for contradiction that $\spec(G') \cap \spec(G\setminus G') = \emptyset$. Let $p^{s} = \exp(G')$ and let $A_{i} = \{g \in G \mid |g| = p^{s+i}\}$. Since $G$ is non-abelian, it follows that $s \geqslant 1$, so  $\spec(G') = \{1,p,\ldots,p^s\}$. It suffices to obtain an element $g \in G \setminus G'$ with $|g|=p$. If $a_{1} \in A_{1}$, then the subgroup $\langle a_{1}, G' \rangle$ is finite \cite[14.5.8]{Rob}. Thus there exists $a_{2} \in C_{G}(a_1)$ such that
$$a_{2} \notin \langle a_{1}, G' \rangle.$$
Since $\langle a_{2}, a_{1}, G' \rangle$ is finite, there exists $a_{3} \in \centralizer(a_{2}) \cap \centralizer(a_{1})$ such that
$$a_{3} \notin \langle a_{2}, a_{1}, G' \rangle.$$ 
Repeating this process, we get $a_{1}, a_{2}, a_{3}, ..., a_{n}, ... \in G \setminus G'$ such that $a_{i} \in \centralizer(a_{i-1}) \cap ... \cap \centralizer(a_{1})$ and 
$$a_{i} \notin \langle a_{i-1}, ..., a_{1}, G' \rangle$$
for any $i \in \mathbb{N}$. Thus $C:= \langle a_{1}, a_{2}, a_{3}, ..., a_{n}, ...\rangle$ is an infinite abelian group. 

Since $G$ has finite exponent, it follows that $C \cap A_{i}$ is infinite for some positive integer $i$. Set $t=\min \{i \in \mathbb{N} \mid C \cap A_i \ \text{is infinite}\}$. Then the set
$$\{a \in C \cap A_{t} \mid \, a^{p} = b \}$$
is infinite for some $b$ with order $p^{s+t-1}$. Choose 
$$a_{i_{1}}, ..., a_{i_{|b|}} \in \{a \in C \cap A_{t} \mid \, a^{p} = b \}.$$ 
Then
$$(a_{i_{1}} ... a_{i_{|b|}})^{p} = a_{i_{1}}^{p} ... a_{i_{|b|}}^{p} = b^{|b|} = 1$$
and
$$a_{i_{1}} ... a_{i_{|b|}} \notin G',$$
a contradiction. Now, it remains to exclude the case when $\exp(G') > \max \{\spec(G \setminus G')\}$. In particular, there are an element $x \in G'$ and a $p$-power $t$ such that $|x^t| \in \spec(G \setminus G')$. Set $h = x^t$. Consequently, there exists an element $g \in G \setminus G'$ such that $|h| = |g|$, which completes the proof.
\end{proof}

We obtain a bound of the exponent of certain infinite non-abelian $\FC$-group in terms of the number of automorphism orbits (see Examples \ref{ex.H} and \ref{ex.n(2)}, below).   

\begin{cor} \label{Cor3.3}
Let $G$ be an infinite non-abelian $\FC$-group with finitely many automorphism orbits. Suppose that $G$ is a $p$-group. Then $\exp(G)$ divides $p^{\omega(G) - 2}$.
\end{cor}

\begin{proof}
Suppose that $\exp(G) \geq p^{\omega(G) - 1}$. Set $g \in G$ such that $|g| = p^{\omega(G) - 1}$. We can deduce that 
\begin{equation} \label{eq:cor.p}
G = \{1\} \uplus g^{\Aut(G)} \uplus (g^{p})^{\Aut(G)} \uplus ... \uplus (g^{p^{\omega(G) - 2}})^{\Aut(G)}.
\end{equation}
By Theorem A, the derived subgroup $G'$ is finite. By Lemma \ref{Lemma 4.2}, there exists an element $h$ of $G \setminus G'$ such that $|h| \in \spec(G')$. In particular, we present two elements of the same order which are in different automorphism orbits, contrary to (\ref{eq:cor.p}). The result follows. 
\end{proof}

\begin{lem} \label{lem.2-groups}
Let $G$ be an infinite non-abelian $\FC$-group. Assume that $G$ is a $2$-group. Then $\omega(G) \geqslant 4$.  
\end{lem}

\begin{proof}
If $\omega(G)=2$, then $G$ is abelian (by Lemma \ref{lem.abelian}). There is no loss of generality in assuming that $\omega(G)=3$. By Corollary \ref{Cor3.3}, $\exp(G)$ divides $4$. Since $G$ is non-abelian, we conclude that $G'$ is non-trivial, so there are elements $h \in G'$ and $g \in G \setminus G'$ such that $|h|=|g|=2$ (by Lemma \ref{Lemma 4.2}). As $\omega(G)=3$, we have $\exp(G)=2$, which is impossible. The proof is complete.      
\end{proof}

The bound on the number of orbits in the above result cannot be improved (see Example \ref{ex.n(2)}, below).

\begin{lem} \label{lem.3.4}
Let $G$ be an infinite non-abelian $\FC$-group. 
\begin{itemize}
\item[(a)] If $\omega(G) = 3$, then $G$ is nilpotent of class $2$ and its exponent $\exp(G)$ is $p$, for some odd prime $p$;
\item[(b)] If $\omega(G) = 4$, then $G$ is nilpotent of class $2$ and the exponent $\exp(G)$ divides $p^2$, for some prime $p$.   
\end{itemize}
\end{lem}
\begin{proof}
$(a).$ By Theorem A, $G'$ is finite. By Lemma \ref{lem.center}, the center $\cent(G)$ is nontrivial. Consequently, $G' = \cent(G)$ and $\omega(\cent(G))=2$, so $\cent(G)$ is a finite elementary abelian $p$-group, for some prime $p$. Since $G$ is nilpotent of class at most $2$ and $\cent(G)$ has exponent $p$, it follows that $G$ is a $p$-group (cf. \cite[5.2.22]{Rob}). Combining Corollary \ref{Cor3.3} and Lemma \ref{lem.2-groups} we deduce that the exponent $\exp(G) = p$ for some odd prime $p$. \\

\noindent $(b).$ We first prove that $G$ is nilpotent of class $2$.

By Theorem A, $G'$ is finite. We have $1 < \omega(G') < 4$. Suppose that $\omega(G') = 3$. Since $\cent_{2}(G)$ is an infinite subgroup, we deduce that there is an element $g \in \cent_{2}(G) \setminus G'$. Thus $G = G' \cent_{2}(G)$ because that characteristic subset meets each orbit under $\Aut(G)$. Moreover, if $\cent_{2}(G) < G$, then there exists $h \in G' \setminus \cent_{2}(G)$. Hence $(hg)^{\Aut(G)}$ does not belong to $G'$ and is not in $\cent_{2}(G)$; which is impossible, because $\omega(G) = 4$. Thus $\cent_2(G) = G$ if $\omega(G')=3$. Now we can assume that $\omega(G') = 2$. Consequently, the derived subgroup $G'$ is a finite elementary abelian $p$-group, for some prime $p$. If $G' \cap \cent(G) = \{ 1 \}$, then $G = G' \cent(G)$ and $G$ is abelian, a contradiction. So $G' \cap \cent(G)$ meets the unique nontrivial orbit of $\Aut(G)$ in $G'$, and $G' \leqslant \cent(G)$ follows. 

Since $2 \leq \omega(\cent(G)) \leq 3$, we conclude that $\cent(G)$ is a $p$-group. It follows that $G$ is a $p$-group (\cite[5.2.22]{Rob}). By Corollary \ref{Cor3.3}, $\exp(G) \leq p^{2}$. 
\end{proof}

The following lemma is a straightforward consequence of \cite[2.1]{BDG}.   

\begin{lem} \label{lem.Z}
Let $G$ be a finite non-soluble group with $\omega(G) \leqslant 6$. Then the center $\cent(G)$ is trivial.  
\end{lem}

We are now in a position to prove Theorem B. 

\begin{proof}[Proof of Theorem B]
$(a)$ By Lemmas \ref{lem.abelian} and \ref{lem.3.4}, $G$ is nilpotent. \\

\noindent $(b)$ Assume that $G$ is an infinite $\FC$-group with $\omega(G) \leqslant 7$. We claim that $G$ is soluble.  

Assume that $G$ is non-soluble. Then, $G'$ is also non-soluble. By Theorem A, $G'$ is finite. By  \cite[Theorem 1]{LD}, $\omega(G') \geqslant 4$. Since $G'$ is a proper characteristic subgroup, we deduce that $\omega(G') \leqslant 6$ under the action of $\Aut(G)$ and so, $G' \cap \cent(G) = 1$ (Lemma \ref{lem.Z}). By Lemma \ref{lem.center}, the center $\cent(G)$ is nontrivial. It follows that $\omega(G) \geqslant \omega(G' \cent(G)) \geqslant 8$ (Lemma \ref{lemma2}), which contradicts $\omega(G)\leqslant 7$. \\ 

\noindent $(c)$ Now assume that $G$ is an infinite non-soluble $\FC$-group with $\omega(G)=8$. Then $G'$ is not soluble, and $\cent(G)$ is not trivial (by Lemma \ref{lem.center}). Suppose that $G' \cap \cent(G) \neq 1$. It follows that $G/(G' \cap \cent(G))$ is a $\BFC$-group with $\omega(G/(G' \cap \cent(G))) \leqslant 7$. We reach the contradiction that $G$ is soluble. So $G' \cap \cent(G)$ is trivial, and the characteristic subgroup $G'\cent(G)$ of $G$ is a direct product $G' \times \cent(G)$. Now $\omega(G') \geqslant 4$ (cf. \cite[Theorem 1]{LD}) yields $8 \leqslant \omega(G' \times \cent(G)) = \omega(G') \cdot \omega(\cent(G)) \leqslant \omega(G)$ (by Lemma \ref{lemma2}). Thus $\omega(G)=8$ enforces $G = G' \times \cent(G)$ and $\omega(\cent(G))=2$. Then $\omega(G')=4$ gives $G' \cong \Alt_5$ (cf. \cite[Theorem 1]{LD}), completing the proof.      
\end{proof}

\begin{ex} \label{ex.5}
In a certain sense the above result cannot be improved. \\

\noindent {\bf (i)} The infinite $\FC$-group $\Sym_{3} \times H_{2}$, where $H_{2}$ is an infinite elementary abelian $2$-group, has five orbits under the action of $\Aut(G)$. For each $h \in H_2$, there exists a homomorphism $\varphi_h$ from $\Sym_3$ to $H_2$ mapping $(12)$ to $h$. Putting $\alpha_h(sv):=s(v\varphi_h(s))$ for $s \in \Sym_3$ and $v \in H_2$ defines an automorphism of $\Sym_{3} \times H_2$ mapping $(12)$ to $(12)h$. Therefore the elements $(1 \, 2)$, $(1 \, 2 \, 3)$, $h$, $(1 \, 2 \, 3)h$ are the representatives of nontrivial orbits of $G$. \\

\noindent {\bf (ii)} The infinite non-soluble $\FC$-group $\Alt_5 \times \mathbb{Q}$ has exactly eight automorphism orbits (by Lemma \ref{lemma2}).

\end{ex}

In \cite{S1} Stroppel proposed the following problem (Problem 2.5): classify all the finite non-solvable groups $G$ with $\omega(G) \leq 6$. We answered Stroppel's question by the following result: 

\begin{thm} (\cite[Theorem 1]{BDG})  \label{thm.nsoluble}
Let $G$ be a finite non-soluble group with $\omega(G) \leqslant 6$. Then $G$ is isomorphic to one of $\PSL(2,\mathbb{F}_4)$, $\PSL(2,\mathbb{F}_7)$, $\PSL(2,\mathbb{F}_9)$, $\PSL(3,\mathbb{F}_4)$, $\ASL(2,\mathbb{F}_4)$.  
\end{thm}

Now, we extend Theorem \ref{thm.nsoluble} to the class of infinite $\FC$-groups.

\begin{cor}
Let $G$ be a non-soluble $\FC$-group with $\omega(G) \leqslant 6$. Then $G$ is finite. Moreover, $G$ is isomorphic to one of $\PSL(2,\mathbb{F}_4)$, $\PSL(2,\mathbb{F}_7)$, $\PSL(2,\mathbb{F}_9)$, $\PSL(3,\mathbb{F}_4)$, $\ASL(2,\mathbb{F}_4)$.
\end{cor}
\begin{proof}
By Theorem B, all infinite $\FC$-group with at most six automorphism orbits are soluble. Since $G$ is non-soluble, it follows that $G$ is finite. Consequently, $G$ is isomorphic to one of the groups indicated in Theorem \ref{thm.nsoluble}. The proof is complete.     
\end{proof}

\section{Proof of Theorem C}
Recall that a finite group $K$ is said to be quasisimple if $K = K'$ and $K/\cent(K)$ is simple. 

\begin{thmC}
Let $G$ be an infinite non-soluble $\FC$-group with at most eleven automorphism orbits. Then $G$ is central-by-finite.  
\end{thmC}

\begin{proof}
Since $G$ is non-soluble, it follows that the derived subgroup $G'$ is non-soluble. By Theorem A, the derived subgroup $G'$ is finite. Thus $\omega(G') \geqslant 4$ (cf. \cite[Theorem 1]{LD}). 

Consider $\omega(G') \in \{4,5,6\}$. It follows that $G' \cap \cent_{2}(G) = 1$ (Theorem \ref{thm.nsoluble}). By Hall's result \cite[14.5.3]{Rob}, the second center has finite index in $G$. If $G$ is not central-by-finite, then $\cent(G)$ is a proper subgroup of $\cent_2(G)$ and  $\omega(\cent_{2}(G)) \geq 3$. Hence, $$\omega(G) \geq \omega(G')\omega(\cent_{2}(G)) \geq 12,$$ 
a contradiction.

Assume that $\omega(G') \in \{7,8,9,10\}$. Set $N = G' \cap \cent_{2}(G)$. Since the derived subgroup  $G'$ is non-soluble, it follows that $\omega (G'/N) \geq 4.$ Consider representatives $g_1N$, $g_2N$ and $g_3N$ for three different non-trivial orbits under $\Aut(G)$ in $G'/N$. Since $[G : \cent_{2}(G)]$ is finite, we can choose $z \in \cent_{2}(G) \setminus G'$. If $zg_{i} \in (zg_{j})^{\Aut(G)}$ for some $1 \leq i \neq j \leq 3$, then there exists $\alpha \in \Aut(G)$ such that
$$zg_{i} = z^{\alpha}g_{j}^{\alpha},$$
that is,
$$g_{i}g_{j}^{-\alpha} = z^{-1}z^{\alpha} \in N.$$
So $g_{i}N \in (g_{j}N)^{\Aut(G)}$, contradicting the choice of representatives. Thus we can conclude that the number of nontrivial orbits of $\Aut(G)$ on $G'/N$ is bonded by $4 - 1 = 3$, because $\omega(G') + |\{z, zg_{1}, zg_{2}, zg_{3}\}| = 11$. It remains to exclude the case when $\omega(G) = 11$,  $G'/N \cong \Alt_{5}$, the derived subgroup $G'$ has seven orbits under the action of $\Aut(G)$, and $\cent_{2}(G) \setminus G'$ is just an orbit. If $\cent(G) \setminus G'$ is a non-empty set, then $G = G' \cent(G)$ and $G$ is central-by-finite. Thus $\cent(G)$ is a subgroup of $G'$ and $N = \cent(G) = \cent(G')$. Since $|G'/N| = 60$, we can deduce that $\omega(\cent(G)) = 2$ under the action of $\Aut(G)$, so $\cent(G)$ is a finite elementary abelian $p$-group for some prime $p$. Since the second center $\cent_{2}(G)$ is infinite and $\cent_2(G) \setminus \cent(G)$ contains just an orbit under the action of $\Aut(G)$, it follows that $\omega(\cent_2(G))=3$ is a $p$-group for some odd prime $p$ (by Lemma \ref{lem.3.4}). Note that $G' = [G', G']$ and $G'/\cent(G') = \Alt_{5}$, then $G'$ is a quasisimple group. Hence the center $\cent(G) = \cent(G')$ is isomorphic to a subgroup of the Schur Multiplier $M(\Alt_5) = C_2$ (cf. \cite[Proposition 2.1.7 and Theorem 2.12.5]{Kar}), which is impossible. The proof is complete.  
\end{proof}

Summarizing, we have the following classification of infinite non-soluble $\FC$-groups with at most eleven automorphism orbits. 

\begin{cor}
Let $G$ be an infinite non-soluble $\FC$-group with $\omega(G) \leqslant 11$. Then $G$ is isomorphic to one of the groups in the set 
$$\{\PSL(2, \mathbb{F}_{q}) \times \cent(G), \, \SL(2, \mathbb{F}_5) \times H_2 \mid q \in \{4, 7, 8, 9\} \},$$
where $\omega(\cent(G)) = 2$ and $H_{2}$ is an infinite elementary abelian $2$-group.
\end{cor}

\begin{proof}
Since the derived subgroup $G'$ is non-soluble, it follows that $\omega(G') \in \{4, 5, 6, 7, 8, 9, 10\}$ (under action of $\Aut(G)$). 

Assume that $\omega(G') \in \{4, 5, 6 \}$. Arguing as in the proof of Theorem B (c), we deduce that  $G$ is isomorphic to one of the following groups: $$\PSL(2, \mathbb{F}_{q}) \times \cent(G),$$ where $q \in \{4, 7, 8, 9\}$ and $\omega(\cent(G)) = 2$. 

Now assume that $\omega(G') \in \{7, 8, 9, 10 \}$.  Arguing as in the proof of Theorem C, we can conclude that the Schur multiplier $M(\Alt_5) = C_2$ and $|\cent(G)| \leqslant 2$. In particular,  $G'$ is isomorphic to one of $\SL(2, \mathbb{F}_5)$ or $\Alt_5$, so $G \cong \SL(2,\mathbb{F}_5)\cent(G)$. In particular, $G \cong \SL(2,\mathbb{F}_5) \times H_2$, where $H_2$ is an infinite elementary abelian $2$-group. The result follows.   
\end{proof}

\section{Neumann's groups}

Groups with three automorphism orbits are called almost homogeneous groups. In \cite{MS1}, M\"aurer and Stroppel prove that the nilpotent almost homogeneous groups with exponent $p$ are generalized Heisenberg group (see also Example \ref{ex.H}, below). In the same article, M\"aurer and Stroppel show that a soluble almost homogeneous group with non-trivial center is nilpotent of class $2$. By Lemma \ref{lem.3.4}, all non-abelian infinite $\FC$-groups $G$ with three automorphism orbits are nilpotent and $\exp(G)=p$ for some odd prime $p$. We present some examples of infinite FC-groups now, we determine the number of orbits under their respective automorphism orbits. The following construction of $\BFC$-groups is due to Neumann \cite{N}.

\begin{ex} (Almost homogeneous $p$-groups) \label{ex.H}
Let $p$ a prime and $\N(p)$ be the group generated by 
$$a_{i}, b_{i}, \, i = 1, 2, ..., n, ...$$
with the relations
$$[a_{i}, b_{j}] = [a_{i}, a_{j}] = 1, \, i \neq j,$$
$$[a_{i}, b_{i}] =c, \, c^{p} = [a_{i}, c] = [b_{i}, c] = 1.$$

If $p$ is odd, then $\N(p)$ is isomorphic to the generalized Heisenberg group $\GH(\V, \mathbb{F}_{p}, \beta)$, where $\V$ is a vector space of countably infinite dimension over the field $\mathbb{F}_{p}$ with $p$ elements (see \cite[page 235]{MS1} for more details). Moreover, $\omega(\GH(\V,\mathbb{F}_{p},\beta))=3$. Its center and derived subgroups are $\langle c \rangle$, thus it is nilpotent of class $2$ and each nontrivial normal subgroup of $\N(p)$ contains the subgroup $\langle c \rangle$. Hence $\{\N(p) \mid \, p \text{ odd prime}\}$ is an infinite family of almost homogeneous $\BFC$-group which are not central-by-finite.  
\end{ex}

It is well known that every group of exponent $2$ is abelian. By Lemma \ref{lem.2-groups}, an infinite non-abelian $2$-group and $\BFC$-group has at least four automorphism orbits. We obtain the following related result. 

\begin{ex} \label{ex.n(2)}
The $\BFC$-group $\N(2)$ has exactly four automorphism orbits.
\end{ex}

\begin{proof}
It is clear that $\N(2)$ is hopfian and every surjective endomorphism is an automorphism.

Let $I$ and $J$ be subsets of the set of positive integers, with $|I| = |J|$, and $\sigma: I \rightarrow J$ a bijection. We collect some facts:

\noindent i) The homomorphism $\alpha_{I}: \N(2) \rightarrow \N(2)$ that extends the map
$$a_{i} \mapsto b_{i}, b_{i} \mapsto a_{i}$$
if $i \in I$ and
$$a_{i} \mapsto a_{i}, b_{i} \mapsto b_{i}$$
if $i \in \mathbb{N} \setminus I$, is an automorphism.

\noindent ii) The homomorphism $\beta_{IJ}: \N(2) \rightarrow \N(2)$ that extends the map
$$a_{i} \mapsto a_{i^{\sigma}}, \, b_{i} \mapsto b_{i^{\sigma}}, \, a_{i^{\sigma}} \mapsto a_{i}, \, b_{i^{\sigma}} \mapsto b_{i}$$
if $i \in I$ and
$$a_{i} \mapsto a_{i}, b_{i} \mapsto b_{i}$$
if $i \in \mathbb{N} \setminus I$, is an automorphism.

\noindent iii) Any element $g \in \N(2)$ can be written in the form
$$g = (a_{i_{1}}b_{i_{1}} ... a_{i_{l}}b_{i_{l}})(a_{j_{1}} ... a_{j_{k}})(b_{t_{1}} ... b_{t_{m}})c^{\delta},$$
where $\delta = 0, 1$ and the set $\{i_{1}, ..., i_{l}, j_{1}, ..., j_{k}, t_{1}, ..., t_{m}\}$ has $l + k + m$ elements. If $l$ is even then $g$ has order $2$ and if $l$ is odd then $g$ have order $4$.

Given $g$ of order $2$, consider $I_{1} = \{i_{1}, ..., i_{l}\}$, $I_{2} = \{j_{1}, ..., j_{k}\}$, $I_{3} = \{t_{1}, ..., t_{m}\}$, $J_{1} = \{1, ..., l\}$, $J_{2} = \{l+1, ..., l+k\}$, $J_{2} = \{l+k+1, ..., l+k+m\}$, and $\sigma_{1} : I_{1} \rightarrow J_{1}$, $\sigma_{2} : I_{2} \rightarrow J_{2}$, $\sigma_{3} : I_{3} \rightarrow J_{3}$ the canonical bijection maps from $I_{r}$ to $J_{r}$, with $r = 1, 2, 3$. Hence
$$h = g^{\alpha_{I_{3}} \beta_{I_{1}J_{1}} \beta_{I_{2}J_{2}} \beta_{I_{3}J_{3}}} = (a_{1}b_{1} ... a_{l}b_{l})(a_{l+1} ... a_{l+k})(a_{l+k+1} ... a_{l+k+m})c^{\delta} =$$
$$= (a_{1}b_{1} ... a_{l}b_{l}) a_{l+1}...a_{s}c^{\delta}$$
with $s = l+k+m$. Now, define the sequences
$$x_{1} = a_{1}b_{1} ... a_{l}b_{l}, x_{2} = b_{1}a_{2}, ..., x_{l} = b_{1}a_{l}, x_{l+1} = a_{l+1}, ...,$$
$$y_{1} = b_{1}, y_{2} = b_{1}b_{2}, ..., y_{l} = b_{1}b_{l}, y_{l+1} = b_{l+1}, ...$$
and
$$u_{1} = a_{1}, ..., u_{l} = a_{l}, u_{l+1} = a_{l+1}...a_{s}c^{\delta}, u_{l+2} = a_{l+2}, ...,$$
$$ v_{1} = b_{1}, ..., v_{l} = b_{l}, v_{l+1} = b_{l+1}, v_{l+2} = b_{l+1}b_{l+2}...,$$
$$v_{s} = b_{l+1}b_{s}, v_{s+1} = b_{s+1}, ....$$
Note that $\{x_{i}, y_{i}; i = 1, 2, 3, ...\}$ and $\{u_{i}, v_{i}; i = 1, 2, 3, ...\}$ are generating sets of $G$ and satisfy its  relations. It follows that the homomorphisms $\alpha: \N(2) \rightarrow \N(2)$ and $\beta: \N(2) \rightarrow \N(2)$ that extend, respectively, the maps
$$x_{i} \mapsto a_{i}, \, y_{i} \mapsto b_{i}$$
and
$$u_{i} \mapsto a_{i}, \, v_{i} \mapsto b_{i}$$
are automorphisms. Therefore
$$h^{\alpha \beta} = a_{1}a_{l+1}.$$
Since there exists $\gamma \in \Aut(\N(2))$ such that $(a_{1}a_{l+1})^{\gamma} = a_{1}$, we conclude that $g$ and $a_{1}$ are in the same orbit.

If $l$ is odd, there exists an automorphism $\mu \in \Aut(\N(2))$ such that $$[(a_{2}b_{2} ... a_{l}b_{l}) a_{l+1}...a_{s}c^{\delta}]^{\mu} = a_{2}$$ and $a_{1}^{\mu} = a_{1}, b_{1}^{\mu} = b_{1}$, that is,
$$h^{\mu} = a_{1}b_{1}a_{2} = (a_{1}a_{2})b_{1}.$$
Note that there exists $\eta \in \Aut(\N(2))$ such that $(a_{1}a_{2})^{\eta} = a_{1}$, $b_{1}^{\eta} = b_{1}$. Thus $h^{\mu\eta} = a_{1}b_{1}$ and $g$ and $a_{1}b_{1}$ are in the same orbit. Therefore,
$$\N(2) = \{1\} \uplus \{c\} \uplus a_{1}^{\Aut(\N(2))} \uplus (a_{1}b_{1})^{\Aut(\N(2))}.$$
\end{proof}

\begin{ex} \label{ex.12}
Let $p$ be an odd prime. Then $\N(p) \times \Alt_5$ is an infinite non-soluble $\BFC$-group with $\omega(\N(p) \times \Alt_5)=12$ (by Lemma \ref{lemma2}). This group is not central-by-finite, so the bound on the number of orbits in Theorem C is a sharp one. 
\end{ex}

\end{document}